\documentclass[reqno]{article}
\usepackage{comment,subeqnarray,amsthm,amsmath,amssymb}
\usepackage{hyperref}
\usepackage{epsfig,color,datetime}
\usepackage{multirow}
\usepackage{colortbl,array}
\usepackage[table]{xcolor}
\usepackage[numbers, square, comma]{natbib}
\usepackage{algpseudocode,algorithm,algorithmicx}
\usepackage{setspace}
\usepackage{graphicx}
\usepackage{subcaption}
\usepackage{tikz} 

\numberwithin{equation}{section}
\numberwithin{figure}{section}
\numberwithin{table}{section}

\def\<{\left\langle}
\def\>{\right\rangle}

\def\NChz2d{{[\mathcal{NC}}^h_0]^2}
\def\NChz{{\mathcal NC}^h_0}

\def\Om{\Omega}
\def\O{\Omega}

\def\Span{\operatorname{Span}}
\def\Tau{{\mathcal T}}

\def\and{\quad\text{and}\quad}

\def\b1{\mathbf 1}

\def\bmu{{\boldsymbol\mu}}

\def\bx{\mathbf x}
\def\bx{\mathbf x}

\def\diam{\operatorname{diam}}
\def\diam{\operatorname{diam}}
\def\dim{\operatorname{dim}\,}
\def\dim{\operatorname{dim}\,}
\def\div{\nabla\cdot}
\def\div{\nabla\cdot}
\def\dsig{\, \operatorname{d\sigma}}

\def\grad{\nabla\,}
\def\grad{\nabla\,}

\def\hat{\widehat}
\def\hat{\widehat}

\def\p{\partial}
\def\p{\partial}

\def\tNChz2d{\widetilde {[\mathcal{NC}}^h_0]^2}
\def\tNChz2d{\widetilde {[\mathcal{NC}}^h_0]^2}

\newcolumntype{x}[1]{>{\centering\hspace{0pt}}p{#1}} 

\newcommand{\bal}{\begin{aligned}}
\newcommand{\bes}{\begin{eqnarray*}}
\newcommand{\be}{\begin{eqnarray}}

\newcommand{\bss}{\begin{subeqnarray*}}
\newcommand{\bs}{\begin{subeqnarray}}

\newcommand{\eal}{\end{aligned}}
\newcommand{\ees}{\end{eqnarray*}}
\newcommand{\ee}{\end{eqnarray}}
\newcommand{\eq}[1]{\begin{eqnarray}\label{#1}}
\newcommand{\ess}{\end{subeqnarray*}}
\newcommand{\es}{\end{subeqnarray}}

\newcommand{\figref}[1]{Figure~\ref{#1}}

\newcommand{\qe}{\end{eqnarray}}

\newcommand{\secref}[1]{Section~\ref{#1}}

\newcommand{\vertiii}[1]{{\left\vert\kern-0.25ex\left\vert\kern-0.25ex\left\vert #1 \right\vert\kern-0.25ex\right\vert\kern-0.25ex\right\vert}}
\newcounter{saveeqn}
\newtheorem{theorem}{Theorem}[section]

\newtheorem{example}[theorem]{Example}

\newtheorem{proposition}[theorem]{Proposition}
\newtheorem{remark}[theorem]{Remark}

\def\til{\widetilde}

\begin{document}

\allowdisplaybreaks

\newif\iflong
\longfalse

\title{Algebraic Multiscale Method for two--dimensional elliptic problems}
\author{Kanghun Cho\thanks{
 Samsung Fire \& Marine Insurance Co., Ltd., 14, Seocho-daero 74-gil,
 Seocho-gu, Seoul 06620,  Korea;
},$\quad$
Imbunm Kim\thanks{Department of Mathematics, 
Seoul National University, Seoul 08826, South Korea;},$\quad$
Raehyun Kim%
\thanks{Department of Mathematics, 
  University of California, Berkeley, CA 94720-3840, USA;
Emails: serein@snu.ac.kr, ikim@snu.ac.kr, rhkim79@math.berkeley.edu, sheen@snu.ac.kr}
,$\quad$
Dongwoo Sheen${}^\dag$
}

\maketitle



\begin{abstract}
We introduce an algebraic multiscale method for two--dimensional
problems. The method uses the generalized multiscale finite element
method based on the quadrilateral nonconforming finite element
spaces. Differently from the one--dimensional algebraic multiscale method, we apply the dimension reduction techniques to construct multiscale basis functions. Also moment functions are considered to impose continuity between local basis functions. Some representative numerical results are presented.
\end{abstract}


%
%

\section{Introduction}
In this paper we propose an algebraic multiscale method for
two--dimensional elliptic problems. This is an extension of our
previous work for the one--dimensional case \cite{cho2022algebraic1d}.
We consider the multiscale model problem given by
\begin{equation}\label{eq:model_problem}
-\div(\kappa(x) \nabla u)=f \text{ in } \Om,
\end{equation}
where $\kappa$ is a heterogeneous coefficient and $f \in H^{-1}(\Om).$

Assume that we are only given a linear system \eqref{eq:phi-psi-vh}, which is obtained from
a discretization of micro-scale governing equation, but any detailed
information on the coefficient $\kappa$ and the source term $f$ are
not given. In this situation, our target is to construct a macro-scale linear
system
by using the information and finally to find numerical solutions which possess similar properties of the solutions obtained by multiscale methods. 

 Multiscale methods have been actively developed in various manners including heterogeneous multiscale methods \cite{abdulle2009finite, abdulle2012heterogeneous}, multiscale hybridizable discontinuous Galerkin methods \cite{cho2020multiscale, efendiev2015spectral, efendiev2015multiscale}, and multiscale finite element methods \cite{efendiev2009multiscale,
efendiev2000convergence,hou1999convergence}. Here we adopt the
generalized multiscale finite element method(GMsFEM)
\cite{efendiev2013generalized, lee2017nonconforming} based on
nonconforming finite element space  \cite{douglas1999nonconforming,
  park2003p}.
However, the method described in this paper applies for the conforming
GMsFEM as well.

The GMsFE spaces consist of snapshot
function spaces, offline function spaces, and moment function
spaces. First, snapshot functions are obtained by solving
$\kappa-$harmonic problems in each macro element. Then offline
functions are constructed by applying suitable dimension reduction
techniques to snapshot function space. In one--dimensional case we
choose offline functions identical to the snapshot functions since
there are only two snapshot functions in each macro element. However
in higher dimension, it is necessary to apply such techniques since we
have many more snapshot functions which yields huge computational
cost. The moment functions are also needed in order to impose
continuity between local offline functions, which make a remarkable
difference between the 1D and 2D cases.

This paper is organized as follows. In Section 2, we briefly review the nonconforming generalized multiscale finite element method(GMsFEM) based on finite element spaces. Then the algebraic multiscale method for two--dimensional elliptic problem is introduced in Section 3 following the framework of GMsFEM. Section 4 is devoted to energy norm error estimate of the proposed method.
In Section 5, representative numerical results are
presented. Conclusions are given in Section 6.

%
%
\section{Preliminaries}

In this section we briefly review the framework of the generalized
multiscale finite element method(GMsFEM) based on the quadrilateral
nonconforming finite element introduced in \cite{douglas1999nonconforming}, following \cite{lee2017nonconforming}. 
We only consider two--dimensional elliptic boundary problems here, but the framework can be extended to higher dimensional cases and used for other multiscale methods. 

Let $U$ be any open subset of $\mathbb{R}^2$. Denote the seminorm, norm, and inner product of the Sobolev space $H^k(U)$ by $|\cdot|_{k,U}$, $||\cdot||_{k,U}$, and $(\cdot,\cdot)_{k,U},$ respectively. For the space $H^0(U) = L^2(U)$, we abbreviate $(\cdot,\cdot)_{0,U}$ as $(\cdot,\cdot)_{U}$. 
Given $f \in H^{-1}(\Omega)$, consider the following elliptic boundary problem:
\begin{equation}\label{eq:2d-ell}
\left\{
\begin{aligned}
-\div \big( \kappa(\textbf{x}) \grad u \big) = f & \text{ in } \Omega,\\
u = 0 & \text{ on }  \partial \Omega,
\end{aligned}
\right.
\end{equation}
where $\Omega$ is a simply connected polygonal domain in $\mathbb{R}^2$, and $\kappa$ is a highly heterogeneous coefficient. 
The weak formulation of \eqref{eq:2d-ell} is to seek $u \in H_0^1(\Omega)$ such that
\begin{equation}\label{eq:2d-weakform}
a(u,v) = F(v), \quad v\in H_0^1(\Omega)
\end{equation}
where $a(u,v)= \int_\Omega \kappa\grad u \cdot \grad v\, d\bx$ and $F(v)=\int_\Omega fv \,d\bx.$
Let $\mathcal{T}_h:=\bigcup_{j=1}^{N_h}\{T_j\}$ be a family of shape regular triangulations of $\Omega$ and $V_h$ be a finite element basis function space based on $\mathcal{T}_h$. The mesh parameter $h$ is given by
\begin{equation*}
h = \max_{j=1,\cdots,N_h} \diam(T_j).
\end{equation*} 
Let $V_{h,0}$ be the set of all elements in $V_h$, whose DOFs related to the boundary $\partial \Omega$ vanish. Then the finite element approximation of \eqref{eq:2d-weakform} is defined as the solution $u_h \in V_{h,0}$ of the discrete problem
\begin{equation}\label{eq:2d-discreteform}
a_h(u_h,v_h) = F_h(v_h), \quad v_h \in V_{h,0},
\end{equation}
where $a_h(u,v)= \sum_{T_j \in \mathcal{T}_h} \int_{T_j} \kappa\grad u \cdot \grad v \,d\bx$ and $F_h(v)=\sum_{T_j \in \mathcal{T}_h} \int_{T_j} fv\, d\bx.$ 
In GMsFEM, we also need to have another shape regular triangulations $\mathcal{T}^H:=\bigcup_{J=1}^{N^H} \{T^J\}$ of $\Omega$. We suppose that every $T^J \in \mathcal{T}^H$ consists of a connected union of $T_j \in \mathcal{T}_h$, which makes $\mathcal{T}_h$ be a refinement of $\mathcal{T}^H$. Here, and in what follows, we refer two triangulations $\mathcal{T}_h$ and $\mathcal{T}^H$ to micro--scale and macro--scale triangulations, respectively. The macro mesh parameter $H$ is given by 
\begin{equation*}
H = \max_{J=1,\cdots,N^H} \diam(T^J).
\end{equation*} 
Let $V^H$ be a finite element basis function space associated with $\mathcal{T}^H$, and $V^{H,0}$ be the set of all elements in $V^H$, whose DOFs related to $\partial\Omega$ vanish. Then the GMsFE approximation of \eqref{eq:2d-weakform} is equivalent to find $u^H \in V^{H,0}$ such that
\begin{equation}\label{eq:2d-GMsFEM}
a_h(u^H,v^H) = F_h(v^H), \quad v^H \in V^{H,0}.
\end{equation}
\subsection{Framework of nonconforming GMsFEM}
The success in using the GMsFEM depends on the construction of corresponding finite element space. $V^H$ must contain the essential properties of $V_h$ as well as the coefficient $\kappa$, while the dimension of $V^H$ is significantly reduced compared to that of $V_h$.

The GMsFE space $V^H$ is composed of
two components. The first one is the offline function space which is a
spectral decomposition of the snapshot function space,
and used to represent the solution in each macro element.
The second one is the moment function space which is used to impose
continuity between local offline functions.
 Let a micro--scale basis function space $V_h$ be given. For each macro element $T \in \mathcal{T}^H$, denote the restriction of $V_h$ to $T$ by $V_h(T)$. Also denote the set of all macro edges in $\mathcal{T}^H$ by $\mathcal{E}^H:=\bigcup_{J=1}^{N^E}\{E^J\}$, and the set of all interior macro edges by $\mathcal{E}^{H,0}$. Then the process of constructing GMsFE spaces is organized into the following framework:
\begin{enumerate}
\item Construct a snapshot function space $V^{\text{snap}}=\bigcup_{T \in \mathcal{T}^H}V^{\text{snap}}(T)$, where $V^{\text{snap}}(T)$ is a subspace of $V_h(T)$ for each macro element $T \in \mathcal{T}^H.$ In general, $V^{\text{snap}}(T)$ is chosen to be the span of $\kappa-$harmonic functions in $T$. 
\item Construct an offline function space $V^{\text{off}}=\bigcup_{T \in \mathcal{T}^H}V^{\text{off}}(T)$, where $V^{\text{off}}(T)$ is obtained by applying a suitable dimension reduction technique to $V^{\text{snap}}(T)$ for each macro element $T \in \mathcal{T}^H$. We may use generalized eigenvalue decomposition, the singular value decomposition, the proper orthogonal decomposition, and so on.
\item Construct a moment function space $\mathcal{M}_H=\bigcup_{E \in \mathcal{E}^H} \mathcal{M}_H(E)$. $\mathcal{M}_H(E)$ may consist of local $\kappa-$harmonic functions in appropriate neighborhood of $E$. The moment functions are used to glue offline functions through each macro interior edge $E \in \mathcal{E}^{H,0}$.
\item Construct the nonconforming GMsFE spaces $V^H$ and $V^{H,0}$ based on $V^{\text{off}}$ and $\mathcal{M}_H$. They are defined as
\begin{align*}
V^H &= \Big\{\psi \in V^{\text{off}} \,\Big|\, \<[\psi]_E,\zeta\>_E=0,\; \forall \zeta \in \mathcal{M}_H(E), \; \forall E \in \mathcal{E}^{H,0} \Big\}, \\
V^{H,0} &=\Big\{\psi \in V^{\text{off}} \,\Big|\, \<[\psi]_E,\zeta\>_E=0,\; \forall \zeta \in \mathcal{M}_H(E), \; \forall E \in \mathcal{E}^{H} \Big\}.
\end{align*}
Here $[\psi]_E$ stands for the jump of $\psi$ across macro edge $E$.
\end{enumerate} 

\subsection{Notations for the DSSY element and edge-based basis functions}
We implement the rectangular DSSY (Douglas-Santos-Sheen-Ye) nonconforming element to construct micro--scale basis function space $V_h$. Since the DSSY elements are based on the horizontal--type and
vertical--type edges, it is more natural to label the edges and basis functions in these two types. For $j=1,\cdots,N_x$ and
$k=1,\cdots,N_y,$  
let $\O_{jk}$ be the $(j,k)^{\text{th}}$ rectangle with the four vertices
$(x_j,y_k),(x_{j-1},y_k),(x_{j-1},y_{k-1}),$ and $(x_j,y_{k-1}),$ and
vertical edges $e_{jk}, e_{j-1,k},$ and horizontal edges $ f_{jk},
f_{j,k-1}.$
Edge-based basis functions are given respectively. That is, the two basis functions
$\phi_{j,k}$ and $\phi_{j-1,k}$ are associated with the edges $e_{jk}$
and $e_{j-1,k},$ and the two basis functions
$\psi_{j,k}$ and $\psi_{j,k-1}$ associated with the edges $f_{jk}$
and $f_{j,k-1}.$ See \figref{fig:dssy-basis-type} for an illustration.

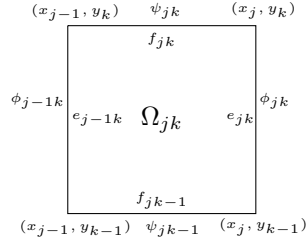
\begin{figure}[htb!]
    \centering
    \begin{tikzpicture}[scale = 0.5]
        \draw (0,5) -- (5,5) -- (5,0) -- (0,0) -- (0,5);
        \node (mid) at (2.5,2.5) {$\O_{jk}$};
        \node (phi1) at (5.2, 5.4) {\tiny $(x_j,y_k)$};
        \node (phi1) at (0.2, 5.35) {\tiny $(x_{j-1},y_k)$};
        \node (phi1) at (5.2, -0.35) {\tiny $(x_{j},y_{k-1})$};
        \node (phi1) at (0.2, -0.4) {\tiny $(x_{j-1},y_{k-1})$};
        \node (phi1) at (4.6, 2.5) {\tiny $e_{jk}$};
        \node (phi1) at (5.5, 3.) {\tiny $\phi_{jk}$};
        \node (phi3) at (0.8, 2.5) {\tiny $e_{j-1k}$};
        \node (phi3) at (-0.8, 3.) {\tiny $\phi_{j-1k}$};
        \node (phi2) at (2.5, 4.6) {\tiny $f_{jk}$};
        \node (phi2) at (2.6, 5.4) {\tiny $\psi_{jk}$};
        \node (phi4) at (2.5, 0.35) {\tiny $f_{jk-1}$}; 
        \node (phi4) at (2.8, -0.4) {\tiny $\psi_{jk-1}$};
    \end{tikzpicture}
    \caption{The basis functions $\phi_{j'k}$ associated with vertical type
      edge $e_{j'k}$ for $j'=j-1,j$ and the
      basis functions $\psi_{jk'}$ associated with horizontal type
      edge $f_{jk'}$ for $k'=k-1,k$ in element $\O_{jk}$ with vertices $(x_{j'},y_{k'})$'s.}
    \label{fig:dssy-basis-type}
\end{figure}

\section{Algebraic Multiscale Method}
In this section, we design an algebraic multiscale method for two--dimensional elliptic problems.
We assume that all the components in $A$ and $b$ in the
  micro--scale linear system $Ax = b$ are known, which 
is constructed by the finite element method to find
  $u_h \in V_h$ such that
\begin{equation}\label{eq:2d-micro-galerkin}
a_h(u_h, v_h) = (f,v_h)_{\O} \quad \forall v_h \in V_h,
\end{equation}
where $a_h(u_h,v_h) = \sum_{j,k} (\kappa \grad u_h,\grad
v_h)_{\O_{jk}}.$ 
Here we do not assume that any a priori knowledge is given for the 
coefficient $\kappa$ and the exterior source term $f$.

On our procedure, we use the GMsFEM under the following assumptions:
\begin{enumerate}
\item The micro--scale mesh is rectangular.
\item The micro--scale linear system is constructed by the DSSY nonconforming
  finite element method.
\item The coefficient $\kappa$ is assumed to be constant on each
  micro element.
\end{enumerate}


Let $V_h=\Span(\{\psi_{jk}\}_{j,k} \bigcup \{\phi_{jk}\}_{j,k})$ be the DSSY nonconforming finite element space associated with $\Tau_h$, where $\psi_{j,k}$ is the DSSY basis functions associated with DOF at midpoint on horizontal micro edge $f_{jk}$, and $\phi_{j,k}$ is the DSSY basis functions associated with DOF at midpoint on vertical micro edge $e_{jk}$ (see \figref{fig:dssy-basis-type}). Then the micro--scale solution $u_h$ is
represented by
\begin{eqnarray}\label{eq:phi-psi-rep}
  u_h = \sum_{j'=0}^{n_x}\sum_{k'=1}^{n_y} \alpha_{j'k'}\phi_{j'k'}
+ \sum_{j'=1}^{n_x}\sum_{k'=0}^{n_y} \beta_{j'k'}\psi_{j'k'}.
\end{eqnarray}
Here, and in what follows, the indices $\alpha$ and $\beta$ stand for
the coefficients for horizontal and vertical edges, respectively.
Test \eqref{eq:2d-micro-galerkin} with $u_h$ represented by
\eqref{eq:phi-psi-rep} against $v_h=\phi_{jk}$ and $v_h=\psi_{jk}$ to obtain
\begin{subeqnarray}\label{eq:phi-psi-vh}
&&\sum_{j'=0}^{n_x}\sum_{k'=1}^{n_y} \alpha_{j'k'}a_h(\phi_{j'k'},\phi_{jk})
+ \sum_{j'=1}^{n_x}\sum_{k'=0}^{n_y} \beta_{j'k'}
  a_h(\psi_{j'k'},\phi_{jk})
  \nonumber\\
&&\qquad\qquad\qquad  = (f,\phi_{jk}),  \quad j =    0,\cdots,n_x, \, k = 1,\cdots, n_y,\\
&&\sum_{j'=0}^{n_x}\sum_{k'=1}^{n_y} \alpha_{j'k'}a_h(\phi_{j'k'},\psi_{jk})
+ \sum_{j'=1}^{n_x}\sum_{k'=0}^{n_y} \beta_{j'k'}
  a_h(\psi_{j'k'},\psi_{jk})\nonumber\\
  &&\qquad\qquad\qquad  = (f,\psi_{jk}),
   \quad j =    1,\cdots,n_x,\, k = 0,\cdots, n_y.
\end{subeqnarray}
Taking into account of the supports of basis functions, we get the
following linear system for the micro-scale elliptic problem:
\begin{subeqnarray}\label{eq:lin-sys}
    &&A^{\alpha,\alpha}_{jkjk}\alpha_{jk}+
    A^{\alpha,\alpha}_{j-1kjk}\alpha_{j-1k}+
    A^{\alpha,\alpha}_{j+1kjk}\alpha_{j+1k}\nonumber\\
    &&+A^{\beta,\alpha}_{jkjk}\beta_{jk}+
    A^{\beta,\alpha}_{jk-1jk}\beta_{jk-1}+
    A^{\beta,\alpha}_{j+1kjk}\beta_{j+1k}+
    A^{\beta,\alpha}_{j+1k-1jk}\beta_{j+1k-1}\nonumber\\
&&    \qquad\qquad\qquad =f^{\alpha}_{jk},
\quad    j =    0,\cdots,n_x, \, k = 1,\cdots, n_y,\\
        &&A^{\beta,\beta}_{jkjk}\beta_{jk}+
    A^{\beta,\beta}_{jk-1jk}\beta_{jk-1}+
    A^{\beta,\beta}_{jk+1jk}\beta_{jk+1}\nonumber\\
    &&+A^{\alpha,\beta}_{jkjk}\alpha_{jk}+
    A^{\alpha,\beta}_{j-1kjk}\alpha_{j-1k}+
    A^{\alpha,\beta}_{jk+1jk}\alpha_{jk+1}+
    A^{\alpha,\beta}_{j-1k+1jk}\alpha_{j-1k+1} \nonumber  \\
    &&\qquad\qquad\qquad
   =f^{\beta}_{jk},
    \quad    j =    1,\cdots,n_x, \, k = 0,\cdots, n_y,
      \end{subeqnarray}
      where
  $A^{\alpha,\alpha}_{j'k'jk} = a_h(\phi_{j'k'},\phi_{jk})$,
  $A^{\alpha,\beta}_{j'k'jk} = a_h(\phi_{j'k'},\psi_{jk})$,
  $A^{\beta,\alpha}_{j'k'jk} = a_h(\psi_{j'k'},\phi_{jk})$,
  $A^{\beta,\beta}_{j'k'jk} = a_h(\psi_{j'k'},\psi_{jk})$,
  $f^\alpha_{jk}=(f,\phi_{jk})$
and  $f^\beta_{jk}=(f,\psi_{jk})$.

A direct computation of the component of the stiffness matrix on
$\O_{jk}=(x_{j-1}, x_{j})\times(y_{k-1}, y_{k})$ gives
\begin{subeqnarray}\label{eq:stif_comp}
(\grad{\phi_{jk}},\grad{\phi_{jk}})_{\O_{jk}}&=&\frac{37}{28}\frac{h_{x_j}}{h_{y_k}}+\frac{65}{28}\frac{h_{y_k}}{h_{x_j}}, \\
(\grad{\psi_{jk}},\grad{\phi_{jk}})_{\O_{jk}}&=&(\grad{\psi_{jk-1}},\grad{\phi_{jk}})_{\O_{jk}}=-\frac{37}{28}\frac{h_{x_j}^{2}+h_{y_k}^{2}}{h_{x_j}h_{y_k}}, \\
(\grad{\phi_{j-1k}},\grad{\phi_{jk}})_{\O_{jk}}&=&\frac{37}{28}\frac{h_{x_j}}{h_{y_k}}+\frac{9}{28}\frac{h_{y_k}}{h_{x_j}}.
\end{subeqnarray}
Analogous components are obtained by replacing $\O_{jk}$ by $\O_{jk-1}.$
Furthermore, we have similar results for $\psi_{jk}$; just $h_{x_j}$ and $h_{y_k}$ are exchanged in \eqref{eq:stif_comp}.
Set $\gamma_{jk}=\frac{h_{y_k}}{h_{x_j}}$. By a direct computation, one gets the following expressions:
\begin{equation*}
    \begin{aligned}
    &A^{\beta,\beta}_{jkjk}=
    \Big(\frac{65}{28}\frac{1}{\gamma_{jk}}+\frac{37}{28}\gamma_{jk}\Big)\kappa_{jk} 
    + \Big(\frac{65}{28}\frac{1}{\gamma_{jk+1}}+\frac{37}{28}\gamma_{jk+1}\Big)\kappa_{jk+1},\\
    &A^{\beta,\beta}_{jk-1jk}=
    \Big(\frac{9}{28}\frac{1}{\gamma_{jk}}+\frac{37}{28}\gamma_{jk}\Big)\kappa_{jk},\\
    &A^{\beta,\beta}_{jk+1jk}=
    \Big(\frac{9}{28}\frac{1}{\gamma_{jk+1}}+\frac{37}{28}\gamma_{jk+1}\Big)\kappa_{jk+1},\\
    &A^{\alpha,\beta}_{jkjk}=
    A^{\alpha,\beta}_{j-1kjk}=
    -\frac{37}{28}\Big(\frac{1}{\gamma_{jk}}+\gamma_{jk}\Big)\kappa_{jk},\\
    &A^{\alpha,\beta}_{jk+1jk}=
    A^{\alpha,\beta}_{j-1k+1jk}=
    -\frac{37}{28}\Big(\frac{1}{\gamma_{jk+1}}+\gamma_{jk+1}\Big)\kappa_{jk+1}.
    \end{aligned}
\end{equation*}

First we need to deduce the coefficient values $\kappa_{jk}$ and mesh sizes
$h_{x_j}, h_{y_k}$ from the micro--scale linear system \eqref{eq:lin-sys}.
The result is formulated as the following proposition.
\begin{proposition}
  $\kappa_{jk}$ and $h_{x_j}, h_{y_k}$ can be determined from the
  linear system \eqref{eq:lin-sys}:
  \begin{subeqnarray}\label{eq:coeff-hx-hy}
        \kappa_{jk}&=&-(A^{\beta,\beta}_{jk-1jk}+A^{\alpha,\beta}_{jkjk})
        \gamma_{jk},\\
        {h_{x_{j}}} &=& \frac1{ \sum_{k=1}^{n_{y}} \gamma_{jk}},\,
        {h_{y_{k}}} = \frac1{\sum_{j=1}^{n_{x}} \frac{1}{\gamma_{jk}}},\\
        \gamma_{jk}&=&\sqrt{\frac{28}{37}\Big(\frac{A^{\alpha,\beta}_{jkjk}}{A^{\beta,\beta}_{jk-1jk}+A^{\alpha,\beta}_{jkjk}}\Big)-1 }.
      \end{subeqnarray}
\end{proposition}
\begin{proof}
At each rectangular elements, except for the 4 corner elements, we can derive at least two information about $\kappa_{jk}$ from the stiffness matrix. One is $A^{\alpha,\alpha}_{j-1kjk}$ or $A^{\beta,\beta}_{jk-1jk}$ and the other is one of $A^{\alpha,\beta}_{jkjk}$, $A^{\alpha,\beta}_{jkjk-1}$, $A^{\alpha,\beta}_{j-1kjk}$, $A^{\alpha,\beta}_{j-1kjk-1}.$ 
 For example, when we have $A^{\beta,\beta}_{jk-1jk}$ and $A^{\alpha,\beta}_{jkjk}$ for left vertical element, we can derive the following equalities:
\begin{equation*}
    A^{\beta,\beta}_{jk-1jk}+A^{\alpha,\beta}_{jkjk}=
    -\frac{\kappa_{jk}}{\gamma_{jk}}, \qquad
        \frac{A^{\alpha,\beta}_{jkjk}}{A^{\beta,\beta}_{jk-1jk}+A^{\alpha,\beta}_{jkjk}}=
    \frac{37}{28}(1+\gamma^{2}_{jk}).
\end{equation*}
Hence $\kappa_{jk}$ and $\gamma_{jk}$ is derived by
\begin{equation*}
      \kappa_{jk}=-(A^{\beta,\beta}_{jk-1jk}+A^{\alpha,\beta}_{jkjk})\, \gamma_{jk},\quad
    \gamma_{jk}=\sqrt{\frac{28}{37}\Big(\frac{A^{\alpha,\beta}_{jkjk}}{A^{\beta,\beta}_{jk-1jk}+A^{\alpha,\beta}_{jkjk}}\Big)-1 }.
\end{equation*}
Other cases follows in a similar way.

At the corner, we cannot get the value of $A^{\alpha,\alpha}_{j-1kjk}$ or $A^{\beta,\beta}_{jk-1jk}$ from the stiffness matrix. That is, there is only one valid information about
$\kappa_{jk}$ and $\gamma_{jk}$.
In this case, we need the ratio information from adjacent elements and this is the reason why we adopt the rectangular mesh. First, we can derive $\gamma_{jk}$ using following relation about ratio,
$\gamma_{jk}=\frac{\gamma_{jk+1}\gamma_{j+1k}}{\gamma_{j+1k+1}}$.
Since the above formula are valid at every micro element except corners, three of $\gamma_{jk}$, $\gamma_{jk+1}$, $\gamma_{j+1k}$, $\gamma_{j+1k+1}$ are known and the unknown one would be determined by the ratio information.
Then $\kappa_{jk}$ can be easily derived from the ratio information and one of $A^{\alpha,\beta}_{jkjk}$, $A^{\alpha,\beta}_{jkjk-1}$, $A^{\beta,\alpha}_{jkjk}$, $A^{\beta,\alpha}_{jkj-1k}$. Now we have every $\kappa_{jk}$ and $\gamma_{jk}$ value across all elements, and mesh sizes $h_{x_j}$ and $h_{y_k}$ are determined by following equations:
\begin{equation*}
      \sum_{k=1}^{n_{y}} \gamma_{jk} = \frac{\sum_{k=1}^{n_{y}} h_{y_{j}}}{h_{x_{j}}} = \frac{1}{h_{x_{j}}},\quad
   \sum_{j=1}^{n_{x}} \frac{1}{\gamma_{jk}} = \frac{\sum_{j=1}^{n_{x}}
     h_{x_{j}}}{h_{y_{k}}} = \frac{1}{h_{y_{k}}}.
  \end{equation*}
  This completes the proof.
\end{proof}


\subsection{Construction of GMsFE spaces}
In this section, we present the detailed procedure for constructing
GMsFE spaces using the approximated
values $\kappa_{jk},h_{x_j},$ and $ h_{y_k},$ obtained as in \eqref{eq:coeff-hx-hy}.
Denote by $\phi_{jk}^T$ and $\psi_{jk}^T$ the basis functions on macro
element $T \in \mathcal{T}^H$ associated with macro vertical and horizontal edges, respectively.

\
\subsection{Snapshot function space $V^{\text{snap}}$}\label{sec:snap}
We first construct local snapshot function space $V^{\text{snap}}(T)$ in each macro element $T \in \mathcal{T}^H$. Since snapshot functions are used to compute multiscale basis functions, we may choose $V^{\text{snap}}(T)$ as all micro--scale basis functions in $T$. 
Or smaller space such as the span of $\kappa-$harmonic functions in $T$ can be considered to reduce the cost of constructing $V^H$. 

Let $\widetilde{u}_l^T \in V_h(T)$ be the solutions of following local $\kappa-$harmonic problems:
\begin{equation}\label{eq:snapshot}
\left\{
\begin{aligned}
-\div \big( \kappa(\bx) \grad \widetilde{u}_l^T \big) = 0 & \text{ in } T,\\
\widetilde{u}_l^T = \delta_l^T & \text{ on }  \p T,
\end{aligned}
\right.
\end{equation}
where $\delta_l^T \in V_h(T)$ is the function which equals to one for the $j-$th micro--scale mesh DOF on $\partial T$ and zeros for the other DOFs on $\partial T$. Observe that $\delta_l^T$ is one of $\phi_{0k}^T,\phi_{n_x^T k}^T, \psi_{j0}^T,\psi_{j n_y^T}^T$ for $j=1,\cdots,n_x^T$ and $k=1,\cdots,n_y^T.$ That is, $\widetilde{u}_l^T$ is the solution of
\begin{equation}\label{eq:snapshot-s}
a_T(\widetilde{u}_l^T, v^T) = 0 \quad \forall \, v^T \in V_{h,0}(T)
\end{equation}
satisfying $\widetilde{u}_l^T-\delta_l^T \in V_{h,0}(T),$ 
where $a_T(u,v) = \sum_{T_j \in T} \int_{T_j} \kappa\grad u \cdot \grad v \,d\bx$.
 
Let $\widetilde{u}_l^T \in V_h(T)$ is represented by 
\begin{equation*}
\widetilde{u}_l^T = \sum_{j'=0}^{n_x^T}\sum_{k'=1}^{n_y^T} \alpha_{j'k'}^T\phi_{j'k'}^T
+ \sum_{j'=1}^{n_x^T}\sum_{k'=0}^{n_y^T} \beta_{j'k'}^T\psi_{j'k'}^T.
\end{equation*}
Then \eqref{eq:snapshot-s} leads to the following equations by setting $v^T = \phi_{jk}^T$ and $v^T = \psi_{jk}^T$:
\begin{subeqnarray}\label{eq:snapshot-test}
&&\sum_{j'=0}^{n_x^T}\sum_{k'=1}^{n_y^T} \alpha_{j'k'}^T a_T(\phi_{j'k'}^T,\phi_{jk}^T)
+ \sum_{j'=1}^{n_x^T}\sum_{k'=0}^{n_y^T} \beta_{j'k'}^T
  a_T(\psi_{j'k'}^T,\phi_{jk}^T) = 0, \nonumber \\
&&\qquad\qquad\qquad\qquad\qquad\qquad j = 0,\cdots,n_x^T, \, k = 1,\cdots, n_y^T,\\
&&\sum_{j'=0}^{n_x^T}\sum_{k'=1}^{n_y^T} \alpha_{j'k'}^Ta_T(\phi_{j'k'}^T,\psi_{jk}^T)
+ \sum_{j'=1}^{n_x^T}\sum_{k'=0}^{n_y^T} \beta_{j'k'}^T
  a_T(\psi_{j'k'}^T,\psi_{jk}^T) = 0, \nonumber\\
  &&\qquad\qquad\qquad\qquad\qquad\qquad j = 1,\cdots,n_x^T,\, k = 0,\cdots, n_y^T.
\end{subeqnarray}

If we take the supports of basis functions into consideration, we have the following linear system for the snapshot function $\widetilde{u}_l^T:$
\begin{subeqnarray}\label{eq:snapshot-lin-sys}
    &&\til A^{\alpha,\alpha}_{jkjk}\alpha_{jk}+
    \til A^{\alpha,\alpha}_{j-1kjk}\alpha_{j-1k}+
    \til A^{\alpha,\alpha}_{j+1kjk}\alpha_{j+1k}\nonumber\\
    &&+ \til A^{\beta,\alpha}_{jkjk}\beta_{jk}+
    \til A^{\beta,\alpha}_{jk-1jk}\beta_{jk-1}+
    \til A^{\beta,\alpha}_{j+1kjk}\beta_{j+1k}+
    \til A^{\beta,\alpha}_{j+1k-1jk}\beta_{j+1k-1}\nonumber\\
&&    \qquad\qquad\qquad = 0 ,
\quad    j =    0,\cdots,n_x^T, \, k = 1,\cdots, n_y^T,\\
        &&\til A^{\beta,\beta}_{jkjk}\beta_{jk}+
    \til A^{\beta,\beta}_{jk-1jk}\beta_{jk-1}+
    \til A^{\beta,\beta}_{jk+1jk}\beta_{jk+1}\nonumber\\
    &&+\til A^{\alpha,\beta}_{jkjk}\alpha_{jk}+
    \til A^{\alpha,\beta}_{j-1kjk}\alpha_{j-1k}+
    \til A^{\alpha,\beta}_{jk+1jk}\alpha_{jk+1}+
    \til A^{\alpha,\beta}_{j-1k+1jk}\alpha_{j-1k+1} \nonumber  \\
    &&\qquad\qquad\qquad
   = 0,
    \quad    j =    1,\cdots,n_x^T, \, k = 0,\cdots, n_y^T,
      \end{subeqnarray}
where
  $\til A^{\alpha,\alpha}_{j'k'jk} = a_T(\phi_{j'k'}^T,\phi_{jk}^T)$,
  $\til A^{\alpha,\beta}_{j'k'jk} = a_T(\phi_{j'k'}^T,\psi_{jk}^T)$,
  $\til A^{\beta,\alpha}_{j'k'jk} = a_T(\psi_{j'k'}^T,\phi_{jk}^T)$, and
  $\til A^{\beta,\beta}_{j'k'jk} = a_T(\psi_{j'k'}^T,\psi_{jk}^T)$. 
  Since each component of the system \eqref{eq:snapshot-lin-sys} can
  be computed from the approximate values of $\kappa_{jk}$ and
  $h_{x_j}, h_{y_k}$ as given in \eqref{eq:coeff-hx-hy},
  we can compute the snapshot function $\widetilde{u}_l^T$ from the
  matrix components $A^{\alpha,\beta}_{jklm}$'s and
  $A^{\beta,\beta}_{jklm}$'s in the linear system \eqref{eq:lin-sys}.
  
  Denote the number of all snapshot functions in $T$ by $\mathcal{N}^{\text{snap}}(T)$ and zero extension of $\widetilde{u}_l^T$ outside $T$ by $u_l^T$. Then the local snapshot function space $V^{\text{snap}}(T)$ is defined as the space spanned by $u_l^T$:
\begin{equation*}
V^{\text{snap}}(T)= \Span\Big\{u_l^T\in V_h(T) \,\Big|\, j=1,\cdots,\mathcal{N}^{\text{snap}}(T)\Big\}.
\end{equation*}
Then the snapshot function space $V^{\text{snap}}$ is defined as the union of such local snapshot function spaces:
\begin{equation*}
V^{\text{snap}} = \bigcup_{T \in \mathcal{T}^H} V^{\text{snap}}(T).
\end{equation*}

\subsubsection{Oversampling technique}
The oversampling technique reduces the resonance error caused by wrong
(local) boundary condition $\delta_l^T$. It consist of
the restriction 
 of the solution $u_l^{T^+}$ of \eqref{eq:snapshot} on an extended
 region $T^+$
to the original domain $T.$
We denote the local oversampled snapshot function space by $V^{\text{snap},+}(T^+)$, which is defined as
\begin{equation*}
V^{\text{snap},+}(T^+)= \Span\Big\{u_l^{T^+}\in V_h(T^+) \,\Big|\, j=1,\cdots,\mathcal{N}^{\text{snap},+}(T^+)\Big\}.
\end{equation*}
Then the oversampled snapshot function space $V^{\text{snap},+}$ is given as follows: 
\begin{equation*}
V^{\text{snap},+} = \bigcup_{T \in \mathcal{T}^H} V^{\text{snap},+}(T^+).
\end{equation*}
 Again we remark that the oversampled snapshot function space is
 completely characterized by the information on the linear system \eqref{eq:lin-sys}.

\subsection{Offline function space $V^{\text{off}}$}
An offline function space $V^{\text{off}}$ is obtained by applying
a suitable dimension reduction technique to the snapshot function
space $V^{\text{snap}}$.
For example, we may use the generalized eigenvalue decomposition.
For each macro element $T \in \mathcal{T}^H$, consider the following spectral problem to find $(\lambda_l^T,u_l^T) \in \mathbb{R} \times V^{\text{snap}}(T):$
\begin{equation}\label{eq:offline}
a_T(u_l^T,v^T) = \lambda_l^T(\kappa u_l^T, v^T)_T, \quad \forall v^T \in V^{\text{snap}}(T).
\end{equation}
Since any function $v \in V^{\text{snap}}(T)$ is represented by
\begin{equation*}
v^T = \sum_{j'=0}^{n_x^T}\sum_{k'=1}^{n_y^T} \alpha_{j'k'}^T\phi_{j'k'}^T
+ \sum_{j'=1}^{n_x^T}\sum_{k'=0}^{n_y^T} \beta_{j'k'}^T\psi_{j'k'}^T,
\end{equation*}
we can construct the linear system of \eqref{eq:offline} and calculate
the offline function $u_l^T$ from the information on the linear system
\eqref{eq:lin-sys}.

We suppose that the eigenvalues are sorted in ascending order as 
\begin{equation*}
0 \leq \lambda_1^T \leq \lambda_2^T \leq \cdots \leq \lambda_{\mathcal{N}^{\text{snap}}(T)},
\end{equation*} 
and the eigenfunctions are normalized by $(\kappa u_l^T,u_l^T)=1.$
Then the local offline function space $V^{\text{off}}(T)$ is defined as the space spanned by a number of dominant eigenfunctions $u_l^T$, which is related to $l-$th smallest eigenvalue $\lambda_l^T.$ We may choose $\mathcal{L}(T)$ eigenfunctions, where $\mathcal{L}(T)$ is considerably small number compared to $\mathcal{N}^{\text{snap}}(T)$. In short, $V^{\text{off}}(T)$ is given by
\begin{equation*}
V^{\text{off}}(T)= \Span\Big\{u_l^T \in V^{\text{snap}}(T)\,\Big|\,l=1,\cdots,\mathcal{L}(T) \Big\},
\end{equation*}
and the offline function space $V^{\text{off}}$ is defined as
\begin{equation*}
V^{\text{off}} = \bigcup_{T \in \mathcal{T}^H} V^{\text{off}}(T).
\end{equation*}

\begin{figure}[t]
\centering
\includegraphics[width=0.4\textwidth]{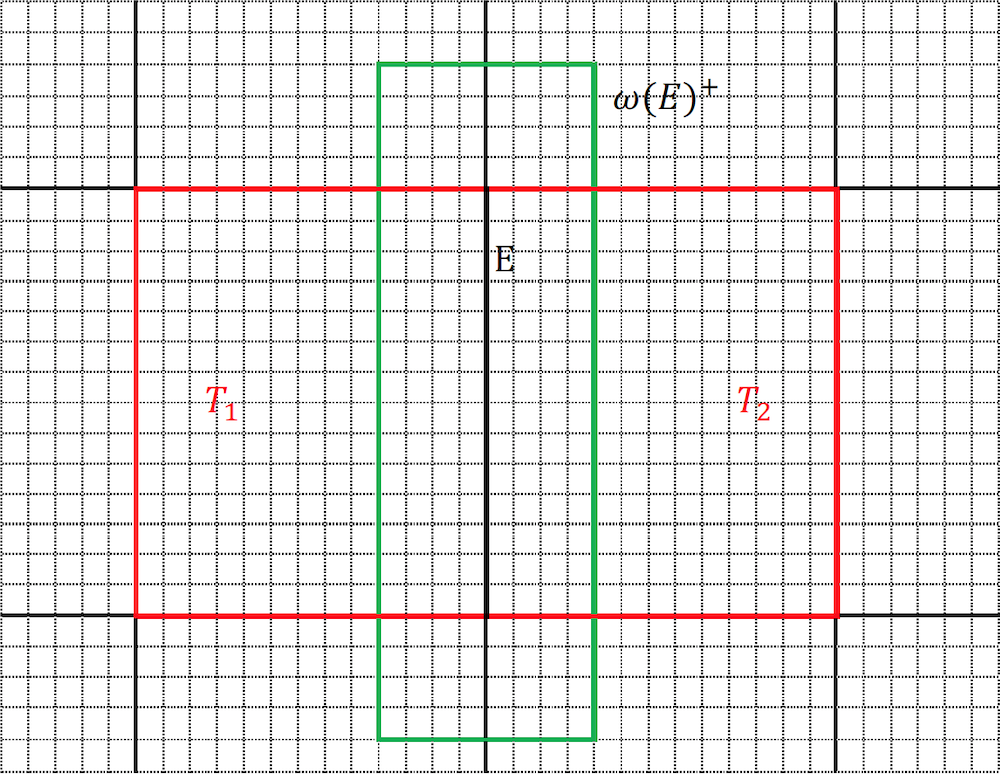}
\caption{Multiscale mesh on $\Omega$. $T_1,T_2$ are macro elements and $\omega(E)^+$ is an oversampled neighborhood of macro edge $E$.}
\label{fig:2d-multi-mesh}
\end{figure}

\subsection{Moment function space $\mathcal{M}_H$}
Since the offline functions are defined independently in each macro
element $T\in \mathcal{T}^H$, we need to glue those functions through
each macro interior edge $E \in \mathcal{E}^{H,0}$. Moment functions
play an important role here, as they are used to impose continuity
between offline functions in neighboring macro elements. On each macro
edge $E$, let $\omega(E)^+$ be an oversampled neighborhood of $E$. As
we construct local snapshot space, the moment function $\zeta_l^E \in
V_h(\omega(E)^+)$ can be obtained by solving the following
local $\kappa-$harmonic problem:
\begin{equation}\label{eq:moment-problem}
\left\{
\begin{aligned}
-\div \big( \kappa(\bx) \grad \zeta_l^E \big) = 0 & \text{ in } \omega(E)^+,\\
\zeta_l^E = \delta_l^E & \text{ on }  \partial \omega(E)^+,
\end{aligned}
\right.
\end{equation}
where $\delta_l^E \in V_h(\omega(E)^+)$ is the function which equals to one for the $l-$th micro--scale mesh DOF on $\partial \omega(E)^+$ and zeros for the other DOFs. We can construct $\zeta_l^E$ as we build the snapshot functions, by replacing $T$ to $\omega(E)^+$ in \secref{sec:snap}.
 We collect the traces of $\zeta_l^E$ on $E$ and perform a singular value decomposition to them. Denote $m(E)$ linearly independent singular vectors by $s_k^E$, where $s_k^E$ is arranged in descending order with respect to its norm:
\begin{equation}\label{eq:moment_singular_vector}
||s_k^E||_E^2 = \mu_k^E, \quad \mu_1^E \geq \mu_2^E \geq \cdots  \geq \mu_{m(E)}^E >0.  
\end{equation}
Then the local moment function space $\mathcal{M}_H(E)$ on $E$ is given by 
\begin{equation*}
\mathcal{M}_H(E)=\Span\Big\{s_k^E \, \Big| \, 1\leq k \leq \mathcal{L}(E)\Big\},
\end{equation*}
and the moment function space $\mathcal{M}_H$ is defined as
\begin{equation*}
\mathcal{M}_H = \bigcup_{E \in \mathcal{E}^H} \mathcal{M}_H(E).
\end{equation*}
It is straightforward to see that
the calculation of the moment function space $\mathcal{M}_H$ is completely 
dependent on the micro--scale linear system \eqref{eq:lin-sys}.
\subsubsection{Another method for constructing moment function space}
We may consider another method for constructing moment function space in order to reduce the computational cost. That is, the moment function space can be made up of the traces of the snapshot functions. For each macro edge $E \in \mathcal{E}^H$, denote the collection of such traces by
\begin{equation*}
\mathcal{M}_h(E) := \Span \Big\{u^T|_E\, \Big| \, u^T \in V^{\text{snap}}(T),E \subset \partial T \Big\}.
\end{equation*}
We perform a singular value decomposition to $\mathcal{M}_h(E)$ and choose the first $\mathcal{L}(E)$ dominant modes of $\mathcal{M}_h(E)$, which span the local moment function space. This method makes us avoid to solve local boundary value problems \eqref{eq:moment-problem}.

\subsection{Nonconforming GMsFE spaces $V^H$ and $V^{H,0}$}
The nonconforming GMsFE spaces $V^H$ and $V^{H,0}$ are defined as
\begin{align}
V^H &= \Big\{u \in V^{\text{off}} \,\Big|\, \<[u]_E,\zeta\>_E=0,\; \forall \zeta \in \mathcal{M}_H(E), \; \forall E \in \mathcal{E}^{H,0} \Big\}, \\
V^{H,0} &=\Big\{u \in V^{\text{off}} \,\Big|\, \<[u]_E,\zeta\>_E=0,\; \forall \zeta \in \mathcal{M}_H(E), \; \forall E \in \mathcal{E}^{H} \Big\},
\end{align}
where $[u]_E$ denotes the jump of $u$ across the edge $E.$
Since $V^{\text{off}}$ and $\mathcal{M}_{H}$ are defined as the union of local function spaces, it is possible to construct $V^H$ and $V^{H,0}$ locally. Let $E \in \mathcal{E}^{H,0}$ be a common macro edge for two macro elements $T_1$ and $T_2$ (see \figref{fig:2d-multi-mesh}.) Suppose that the local moment function space $\mathcal{M}_H(E)$ is constructed from $\kappa-$harmonic functions in $\omega(E):=T_1 \cup T_2$. Then the continuity condition for $u \in V^{\text{off}}(T_1) \cup V^{\text{off}}(T_2)$ imposed by $\mathcal{M}_H(E)$ is given as follows:
\begin{align}\label{eq:cont_cond_local}
\begin{split}
\<[u]_E,\zeta\>_E &= 0,\quad \forall \zeta \in \mathcal{M}_H(E), \\
\<u,\zeta\>_{E'} &=0,  \quad \forall \zeta \in \mathcal{M}_H(E'), \quad \forall E' \subset \partial \omega(E).
\end{split}
\end{align}
Finally we define the local GMsFE space $V^H(\omega(E))$ as
\begin{equation*}
V^H(\omega(E)) = \Big\{u \in V^{\text{off}}(T_1) \cup V^{\text{off}}(T_2)  \,\Big|\, u \text{   satisfies   } \eqref{eq:cont_cond_local} \Big\}.
\end{equation*}
Then the GMsFE space $V^{H,0}$ can be obtained by
\begin{equation*}
V^{H,0} = \bigcup_{E \in \mathcal{E}^{H,0}} V^H(\omega(E)).
\end{equation*}
$V^H$ also can be constructed similarly by considering $E \in \mathcal{E}^H$ in the above argument.
\begin{remark}
It is remarkable that there may exist macro bubble functions $u\in V^{\text{off}}(T)$ on $T$, which satisfy
\begin{equation}\label{eq:bubble_cond}
\<u,\zeta\>_{E'} = 0, \quad \forall \zeta \in \mathcal{M}_H(E'), \quad \forall E' \subset \partial T.
\end{equation}
Denote the space of macro bubble functions on $T_j$ by 
\begin{equation*}
B_H(T) = \Big\{u \in V^{\text{off}}(T) \,\Big|\, u \text{   satisfies   } \eqref{eq:bubble_cond} \Big\}.
\end{equation*}
Then the GMsFE space $V^{H,0}$ is obtained by
\begin{equation*}
V^{H,0} = \Big(\bigcup_{T \in \mathcal{T}^H} B_H(T)\Big) \bigcup \Big(\bigcup_{E \in \mathcal{E}^{H,0}} V^H\big(\omega(E)\big)\Big).
\end{equation*}
\end{remark}

\begin{remark}\label{rmk:dim_gmsfem}
The dimension of GMsFE space $V^{H,0}$ may depend on the dimension of local moment function space. For each macro element $T \in \mathcal{T}^H$, we practically take the dimension of local offline function space as
\begin{equation*}
\mathcal{L}(T) = \sum_{E'\subset\partial T}\mathcal{L}(E').
\end{equation*}
Then the dimension of $V^H\big(\omega(E)\big)$ is given by
\begin{equation*}
\dim \Big(V^H(\omega(E)) \Big) \geq \mathcal{L}(T_1) + \mathcal{L}(T_2) - \mathcal{L}(E) - \sum_{E' \subset \partial \omega(E)} \mathcal{L}(E') = \mathcal{L}(E).
\end{equation*}
If there are no macro bubble functions, it follows that 
\begin{equation*}
\dim(V^{H,0}) = \sum_{E \in \mathcal{E}^{H,0}} \mathcal{L}(E).
\end{equation*}
\end{remark}

\subsubsection{Construction of $b^M$}
Now we have the GMsFE space $V^{H,0}=\Span(\{\psi_{JK}^L\}_{J,K,L} \bigcup \{\phi_{JK}^L\}_{J,K,L})$, where $\psi_{JK}^L$ and $\phi_{JK}^L$ denote the $L$-th multiscale basis function associated with the $JK-$th horizontal macro edge $f_{JK}$ and $JK-$th vertical macro edge $e_{JK}$, respectively. Suppose that $\Omega_{JK} \cup \Omega_{J+1 K}$ is composed of $n_x^{JK} \times n_y^{JK}$ micro--scale elements. Then $\phi_{JK}^L$ is represented by
\begin{equation*}
\phi_{JK}^L = \sum_{j'=0}^{n_x^{JK}}\sum_{k'=1}^{n_y^{JK}} \alpha_{j'k'}^{JK}\phi_{j'k'}^{JK}
+ \sum_{j'=1}^{n_x^{JK}}\sum_{k'=0}^{n_y^{JK}} \beta_{j'k'}^{JK}\psi_{j'k'}^{JK},
\end{equation*}  
where $\phi_{j'k'}^{JK}$ and $\psi_{j'k'}^{JK}$ are the micro--scale basis functions of vertical and horizontal type in $\Omega_{JK} \cup \Omega_{J+1 K}$, respectively.
Therefore it is obvious that the components of $b^M$ can be derived
from the summation of that of micro--scale right hand side $b$, which is
given by $f^{\alpha}_{jk}$ and $f^{\beta}_{jk}$ in \eqref{eq:lin-sys}.
The same argument holds for $\psi_{JK}^L$, which completes the construction of $b^M$.

\section{Numerical results}
\begin{example}
Consider the following elliptic problem: 
\begin{equation}\label{eq:exam-1}
\left\{
\begin{aligned}
-\div \big( \kappa(\bx) \grad u \big) = f & \text{ in } \O,\\
u = 0 & \text{ on }  \p \O,
\end{aligned}
\right.
\end{equation}
where $\O=(0,1)^2$ and $\kappa(\bx)=1+(1+x_1)(1+x_2)+\epsilon\sin(10\pi x_1)\sin(5\pi x_2).$ The source term $f$ is generated by the exact solution
\begin{equation*}
    u(x_1,x_2) = \sin(3\pi x_1)x_2(1-x_2) + \epsilon\sin(\pi x_1/\epsilon)\sin(\pi x_2/\epsilon).
\end{equation*}
\end{example}
We compare numerical results of GMsFEM and AMS(algebraic multiscale
method). We use uniform rectangular meshes and the ratio $H/h$ is
fixed as $10$. The AMS solutions are calculated from the information on
the linear system $Ax=b$ obtained by
the application of the finite
element method based on the DSSY nonconforming element for the
elliptic problem \eqref{eq:exam-1}.

The relative energy and $L^2$ errors for the macro-scale solutions
are reported for various $\epsilon$ in Tables 4.1--3.
We observe almost same error behaviors in both methods. 
\begin{table}[h!]
\centering
\begin{tabular}{|c|c|c|c|c|c|c|}
\hline
\multirow{2}{*}{$\frac{1}{H}$} & \multirow{2}{*}{$\frac{1}{h}$} & \multirow{2}{*}{$\dim(V^{H,0})$} & \multicolumn{2}{c|}{GMsFEM} & \multicolumn{2}{c|}{AMS} \\ \cline{4-7} 
                               &                                &                                  & Rel. Energy   & Rel. $L^2$  & Rel. Energy & Rel. $L^2$ \\ \hline
5                              & 50                             & 400                              & 0.884         & 0.388       & 0.884       & 0.389      \\ \hline
10                             & 100                            & 1800                             & 0.871         & 0.363       & 0.871       & 0.363      \\ \hline
20                             & 200                            & 7600                             & 0.346         & 0.676E-01   & 0.346       & 0.678E-01  \\ \hline
40                             & 400                            & 31200                            & 0.181         & 0.181E-01   & 0.181       & 0.182E-01  \\ \hline
\end{tabular}
\caption{Convergence for $\epsilon=0.1.$}
\label{tab:ex3-2-1-eps01}
\end{table}

\begin{table}[h!]
\centering
\begin{tabular}{|c|c|c|c|c|c|c|}
\hline
\multirow{2}{*}{$\frac{1}{H}$} & \multirow{2}{*}{$\frac{1}{h}$} & \multirow{2}{*}{$\dim(V^{H,0})$} & \multicolumn{2}{c|}{GMsFEM} & \multicolumn{2}{c|}{AMS} \\ \cline{4-7} 
                               &                                &                                  & Rel. Energy   & Rel. $L^2$  & Rel. Energy & Rel. $L^2$ \\ \hline
5                              & 50                             & 400                              & 0.885         & 0.625       & 0.885       & 0.625      \\ \hline
10                             & 100                            & 1800                             & 0.355         & 0.118       & 0.355       & 0.119      \\ \hline
20                             & 200                            & 7600                             & 0.186         & 0.316E-01   & 0.186       & 0.320E-01  \\ \hline
40                             & 400                            & 31200                            & 0.940E-01     & 0.803E-02   & 0.939E-01   & 0.823E-02  \\ \hline
\end{tabular}
\caption{Convergence for $\epsilon=0.2.$}
\label{tab:ex3-2-1-eps02}
\end{table}

\begin{table}[h!]
\centering
\begin{tabular}{|c|c|c|c|c|c|c|}
\hline
\multirow{2}{*}{$\frac{1}{H}$} & \multirow{2}{*}{$\frac{1}{h}$} & \multirow{2}{*}{$\dim(V^{H,0})$} & \multicolumn{2}{c|}{GMsFEM} & \multicolumn{2}{c|}{AMS} \\ \cline{4-7} 
                               &                                &                                  & Rel. Energy   & Rel. $L^2$  & Rel. Energy & Rel. $L^2$ \\ \hline
5                              & 50                             & 400                              & 0.335         & 0.130       & 0.335       & 0.132      \\ \hline
10                             & 100                            & 1800                             & 0.173         & 0.342E-01   & 0.173       & 0.351E-01  \\ \hline
20                             & 200                            & 7600                             & 0.884E-01     & 0.879E-02   & 0.885E-02   & 0.928E-02  \\ \hline
40                             & 400                            & 31200                            & 0.444E-01     & 0.221E-02   & 0.444E-01   & 0.246E-02  \\ \hline
\end{tabular}
\caption{Convergence for $\epsilon=0.5.$}
\label{tab:ex3-2-1-eps05}
\end{table}

\begin{figure}[h!]
\centering
\includegraphics[width=0.8\textwidth]{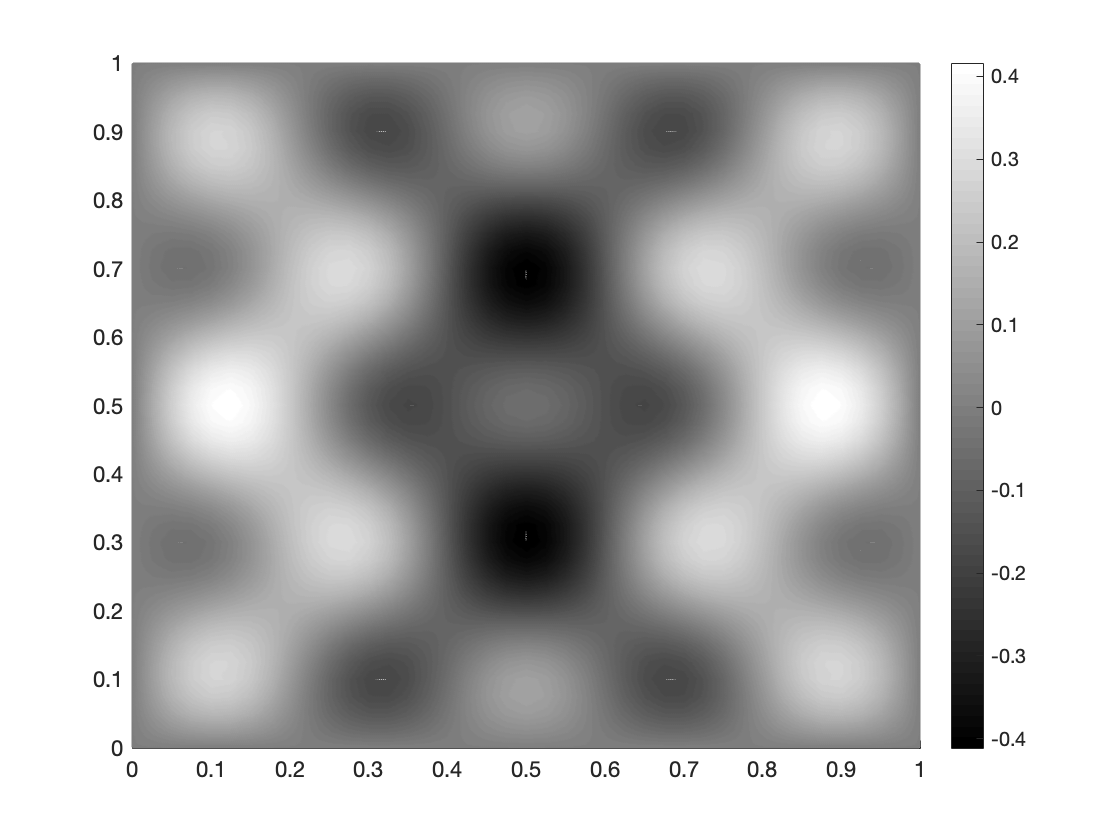}
\caption{Algebraic multiscale solution of $\epsilon=0.2$ when $1/H=40, 1/h=400$.}
\label{fig:alg-400eps02}
\end{figure}

\begin{figure}[h!]
\centering
\includegraphics[width=0.8\textwidth]{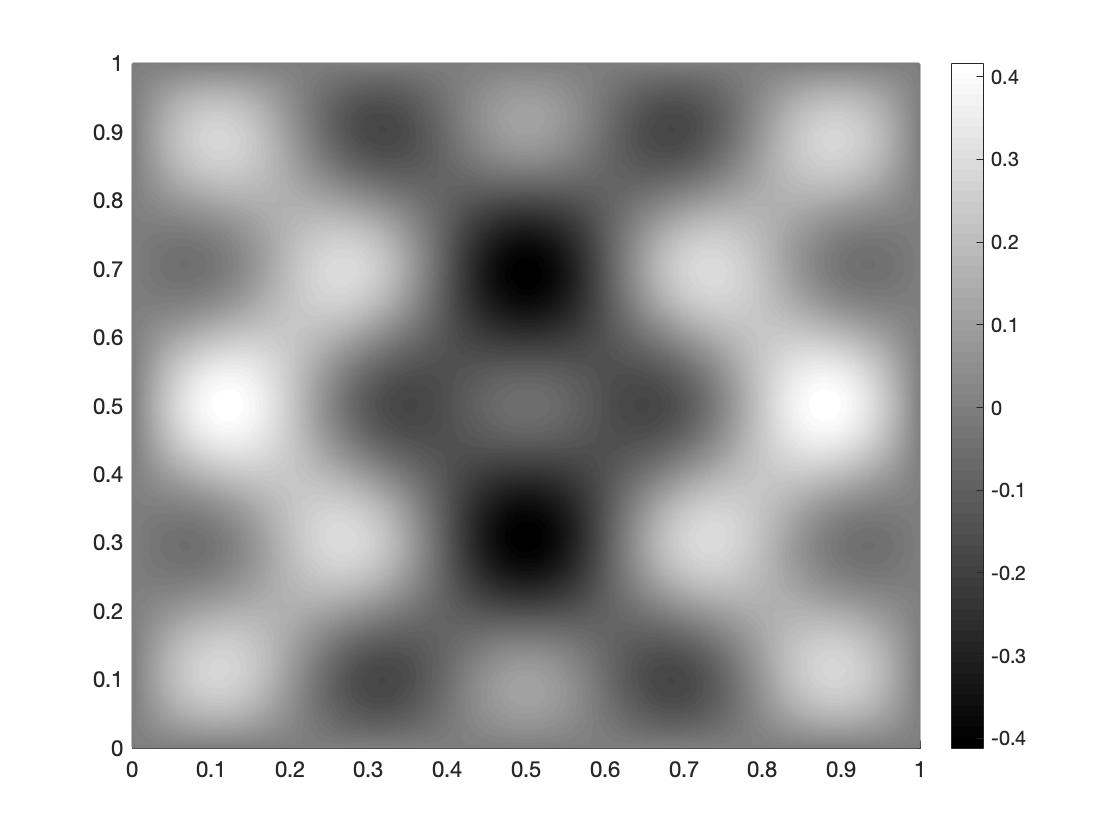}
\caption{Micro--Scale reference solution of $\epsilon=0.2$ when $1/h=400$.}
\label{fig:alg-refeps02}
\end{figure}

%
%
\section{Conclusion}
In this paper, we design the algebraic multiscale method for
two--dimensional elliptic problems. The generalized multiscale finite
element method is used based on the DSSY nonconforming finite element
space. The macro--scale linear system is constructed only using the
algebraic information on the components of the micro--scale system. The
proposed method shows almost identical numerical results with those
obtained by the GMsFEM.

\section*{Acknowledgment}
DS was supported in part by National Research Foundation of
Korea (NRF-2017R1A2B3012506 and NRF-2015M3C4A7065662). 

\normalsize
\bibliographystyle{abbrv}
\bibliography{ams}
\end{document}